\documentclass[reqno,12pt]{amsart}
\usepackage{txfonts}
\usepackage{mathrsfs}
\usepackage{amsmath}
\usepackage{amssymb}

\usepackage{amsthm}
\usepackage{verbatim}
\usepackage{latexsym, bm}

\title[SMOOTH SOLUTION TO HIGHER DIMENSIONAL COMPLEX
PLATEAU PROBLEM]{SMOOTH SOLUTION TO HIGHER DIMENSIONAL COMPLEX
PLATEAU PROBLEM}

\author[Rong Du]{Rong Du$^{\dag}$}
\address{Department of Mathematics\\
Shanghai Key Laboratory of PMMP\\
East China Normal University\\
Rm. 312, Math. Bldg, No. 500, Dongchuan Road\\
Shanghai, 200241, P. R. China} \email{rdu@math.ecnu.edu.cn}

\thanks{$^{\dag}$ The Research Sponsored by the National Natural Science Foundation of China (Grant No. 11471116, 11531007) and Science and Technology Commission of Shanghai Municipality (Grant No. 13dz2260400).}

\theoremstyle{definition}
\newtheorem{theorem}[subsection]{Theorem}
\newtheorem{lemma}[subsection]{Lemma}
\newtheorem{definition}[subsection]{Definition}

\newtheorem{corollary}[subsection]{Corollary}
\newtheorem{remark}[subsection]{Remark}
\newtheorem{conjecture}[subsection]{Conjecture}
\newfont{\drnew}{wncyr10}

\allowdisplaybreaks

\def\dashfill{\leaders\hbox{\hbox to 3.25pt{\hrulefill}\hspace*{2pt}\hbox to 3.25pt{\hrulefill}}\hfill}
\newcommand{\CITE}[1]{{[#1]}}
\let\cite=\CITE

\begin{document}

\begin{abstract}
Let $X$ be a compact connected strongly pseudoconvex $CR$ manifold
of real dimension $2n-1$ in $\mathbb{C}^{N}$.  For $n\ge 3$, Yau solved the complex
Plateau problem of hypersurface type by checking a bunch of Kohn-Rossi cohomology groups in 1981.
In this paper, we generalize Yau's conjecture on some numerical invariant of every isolated surface singularity defined by Yau and the author to any dimension and prove that the conjecture is true for local complete intersection singularities of dimension $n\ge 3$. As a direct
application, we solved complex Plateau problem of hypersurface type for any dimension $n\ge 3$ by checking only one numerical invariant.
\end{abstract}

\maketitle

\vspace{1cm}
\section{\textbf{Introduction}}
The famous classical complex Plateau problem asks which odd
dimensional real sub-manifolds of $\mathbb{C}^N$ are boundaries of
complex sub-manifolds in $\mathbb{C}^N$. Harvey and Lawson \cite{Ha-La} proved in their beautiful seminal
paper that for any compact
connected $CR$ manifold $X$ of real dimension $2n-1$, where $n\ge 2$, in
$\mathbb{C}^N$, there is a unique complex variety $V$ in
$\mathbb{C}^N$ such that the boundary of $V$ is $X$.

The next question is to determine when $X$ is a boundary of a
complex sub-manifold in $\mathbb{C}^N$, i.e. when $V$ is smooth.
Suppose $X$ is a compact connected strongly pseudoconvex $CR$
manifold of real dimension $2n-1$, $n\ge 3$, in the boundary of a
bounded strongly pseudoconvex domain $D$ in $\mathbb{C}^{n+1}$. In
1981, Yau \cite{Ya} showed  that $X$ is a boundary of the complex
sub-manifold $V\subset D-X$ if and only if Kohn-Rossi cohomology
groups $H^{p, q}_{K R}(X)$ are zeros for $1\le q\le n-2$. So Yau's result solved the classical Plateau problem for hypersurface type as $n\ge 3$.

For $n=2$, i.e. $X$ is a 3-dimensional $CR$ manifold, the intrinsic
smoothness criteria for the complex Plateau problem remains unsolved
for over a quarter of a century even for the hypersurface type. The
main difficulty is that the Kohn-Rossi cohomology groups are
infinite dimensional in this case. Let $V$ be a complex variety with
$X$ as its boundary, then the singularities of $V$ are surface
singularities. In \cite{Lu-Ya}, Luk and Yau proved that if $X$ is a
compact strongly pseudoconvex  Calabi-Yau $CR$ manifold of dimension
$3$ contained in the boundary of a strongly pseudoconvex bounded
domain $D$ in $\mathbb{C}^N$ and the holomorphic De Rham cohomology
$H^2_h(X)$ vanishes, then $X$ is a boundary of a complex variety $V$
in $D$ with boundary regularity and $V$ has only isolated
singularities which are all Gorenstein surface singularities with vanishing
$s$-invariant in the interior. Later, in \cite{Du-Ya}, Yau and the author
introduced a new $CR$ invariant $g^{(1,1)}(X)$ which is intrinsic in term of $X$. This new invariant
gives a necessary and sufficient condition for the variety $V$
bounded by $X$ with the holomorphic De Rham cohomology $H^2_h(X)=0$
being smooth.

From \cite{Du-Ya}, we know that $g^{(1,1)}(X)$ is related to the
corresponding invariants of singularities in the variety which $X$
bounds. The main idea for solving the Plateau problem is to show the
strict positivity of the invariant $g^{(1,1)}$ of each normal
singularity. Yau conjectured that $g^{(1,1)}>0$ for every isolated
normal surface singularity (see \cite{Du-Ga}).

\vspace{.5cm} \textbf{Conjecture:} For each isolated normal surface
singularity, the invariant $g^{(1,1)}$ is strictly positive.
\vspace{.5cm}

In \cite{Du-Ya}, we showed that $g^{(1,1)}(X)>0$ for every isolated
normal singularity with $\mathbb{C}^*$-action. It provides some evidence to make one believe the truth of the conjecture.

For non-hypersurface type, $n\ge 3$, Yau's method doesn't work because his method relies heavily on the Tyurina numbers of hypersurface singularities while they cannot be generalized for arbitrary singularities in an effective way. Moreover, numerical conditions on the boundary $X$ can hardly guarantee the normality of the singularities in the interior which is automatically satisfied for isolated hypersurface singularities though. So it is natural to ask the second best thing that when $X$ is a boundary of a variety $V$ which is smooth after normalization. Recently, Gao, Yau and the author generalized the invariant $g^{(1,1)}$ (resp. $g^{(1,1)}(X)$) to
higher dimension as $g^{(\Lambda^n 1)}$ (resp. $g^{(\Lambda^n 1)}(X)$) and showed that if
$g^{(\Lambda^n 1)}(X)=0$, then the interior has at worst finite
number of rational singularities (\cite{D-G-Y}). In particular, if $X$ is Calabi--Yau of real
dimension $5$, then the vanishing of this invariant is equivalent to
give the interior regularity up to normalization. The main method in \cite{D-G-Y} relies Reid's results on the existence of the explicit resolution of $3$ dimensional Gorenstein terminal and canonical singularities and the crucial point is to show the positivity of the numerical invariant $g^{(\Lambda^3 1)}$ for $3$ dimensional isolated Gorenstein singularities. In \cite{D-G-Y}, we generalized Yau's above conjecture to dimension $3$ and proved it holds for $3$ dimensional isolated Gorenstein singularities.

In this paper we generalize Yau's conjecture to any dimension and prove that it is true for isolated local complete intersection singularities of dimension $n\ge 3$. More explicitly, we show that $g^{(\Lambda^n 1)}$ is strictly positive for
every isolated local complete intersection singularity of dimension $n\ge 3$. The crucial point for studying the numerical invariant $g^{(\Lambda^n 1)}$ is to understand holomorphic $1$-forms on the resolution of isolated singularities. So we transfer the original complex Plateau problem to studying holomorphic $1$-forms or holomorphic vector fields dually on singular varieties. Holomorphic vector fields on singular varieties are interesting in their own right (cf. [S-S1], [S-S2] or [B-S-S]). Our method is considering the vanishing order of holomorphic $1$-forms along some special exceptional component on the resolution manifold of isolated singularities mainly. As a direct
application of the strictly positivity of the numerical invariant $g^{(\Lambda^n 1)}$, we can deal with the complex Plateau problem of hypersurface type for any dimension $n\ge 3$ by checking only one numerical invariant.

\vspace{.5cm}\textbf{Main Theorem 1}  \emph{Let $X$ be a strongly
pseudoconvex compact $CR$ manifold of real dimension $2n-1\ge 5$.
Suppose that $X$ is contained in the boundary of a strongly
pseudoconvex bounded domain $D$ in $\mathbb{C}^{n+1}$. Then $X$ is a
boundary of the complex sub-manifold
$V\subset D-X$ with boundary regularity if and only if
$g^{(\Lambda^n 1)}(X)=0$. }
\vspace{.5cm}

The crucial point for the above theorem is by generalizing Yau's conjecture to higher dimension and proving it is true for isolated local complete intersection singularities of dimension $n\ge 3$.

\vspace{.5cm} \textbf{Conjecture:} For each isolated normal singularity of dimension $n$, the invariant $g^{(\Lambda^n 1)}$ is strictly positive.
\vspace{.5cm}

\vspace{.5cm}\textbf{Main Theorem 2}  \emph{Let $V$ be a $n$-dimensional Stein space with $0$ as its only
local complete intersection singular point, then $g^{(\Lambda^n 1)}\ge 1$. }
\vspace{.5cm}

The proof of theorem relies on Shokurov's minimal discrepancy conjecture (\cite{Sh}). Up to now, we only know it holds for isolated local complete intersection singularities when the dimension of singularity is great than $3$ (see \cite{E-M}, \cite{E-M-Y}). If Shokurov's minimal discrepancy conjecture were true for isolated Gorenstein terminal singularities , then our Main Theorem 2 is also true for isolated Gorenstein singularities of dimension $n\ge 3$ and our Main Theorem 1 holds for Calabi-Yau CR manifold of non-hypersurface type.

\vspace{.5cm} In Section 2, we shall recall some basic definitions of a $CR$
manifold. In Section 3, we survey some results of invariant of singularities and CR invariants introduced in \cite{D-G-Y}. Moreover, we
show the strictly positivity of the invariant $g^{(\Lambda^n 1)}$ for isolated local complete intersection singularities of dimension $n\ge 3$. In Section 4, we solve our Main Theorem 1 in this paper.

\section{\textbf{Strongly pseudoconvex $CR$ manifolds}}
We will recall the basic definition of strongly pseudoconvex $CR$
manifolds.  We recommend \cite{Ta} or the preliminaries in \cite{Du-Ya} for the
details.
\begin{definition}
\emph{Let $X$ be a connected orientable manifold of real dimension
$2n-1$. A $CR$ structure on $X$ is a rank $n-1$ subbundle $S$ of
$\mathbb{C}T(X)$ (complexified tangent bundle) such that
\begin{enumerate}
\item
$S\bigcap \bar{S}=\{0\}$,
\item
If $L$, $L'$ are local sections of $S$, then so is $[L, L']$.
\end{enumerate}}
\end{definition}

\begin{definition}
\emph{Let $L_1,\dots, L_{n-1}$ be a local frame of the $CR$
structure $S$ on $X$ so that $\bar{L}_1,\dots,\bar{L}_{n-1}$ is a
local frame of $\bar{S}$. Since $S\oplus  \bar{S}$ has complex
codimension one in $\mathbb{C}T(X)$, we may choose a local section N
of $\mathbb{C}T(X)$ such that $L_1,\dots, L_{n-1},
\bar{L}_1,\dots,\bar{L}_{n-1}$, $N$ span $\mathbb{C}T(X)$. We may
assume that $N$ is purely imaginary. Then the matrix $(c_{ij})$
defined by
\[
[L_i,
\bar{L}_j]=\sum_ka^k_{i,j}L_k+\sum_kb^k_{i,j}\bar{L}_k+c_{i,j}N
\]
is Hermitian, and is called the Levi form of $X$.}
\end{definition}
\begin{remark}
\emph{The number of non-zero eigenvalues and the absolute value of
the signature of $(c_{ij})$ at each point are independent of the
choice of $L_1,\dots, L_{n-1}, N$.}
\end{remark}
\begin{definition}
\emph{$X$ is said to be strongly pseudoconvex if the Levi form is
positive definite at each point of $X$.}
\end{definition}

\begin{definition}
\emph{Let $X$ be a CR manifold of real dimension $2n-1$. $X$ is said
to be Calabi-Yau if there exists a nowhere vanishing holomorphic
section in $\Gamma(\wedge^n\widehat{T}(X)^*)$, where
$\widehat{T}(X)$ is the holomorphic tangent bundle of $X$.}
\end{definition}

\textbf{Remark}:
\begin{enumerate}
\item
Let $X$ be a $CR$ manifold of real dimension $2n-1$ in
$\mathbb{C}^n$. Then $X$ is a Calabi-Yau $CR$ manifold.
\item
Let $X$ be a strongly pseudoconvex $CR$ manifold of real dimension
$2n-1$ contained in the boundary of bounded strongly pseudoconvex
domain in $\mathbb{C}^{n+1}$. Then $X $ is a Calabi-Yau $CR$
manifold.
\end{enumerate}

\section{\textbf{Invariants of singularities and CR invaraints}}

Let $X$ be a compact connected strongly pseudoconvex $CR$ manifold
of real dimension $2n-1$, in the boundary of a bounded strongly
pseudoconvex domain $D$ in $\mathbb{C}^N$. By a result of Harvey and
Lawson, there is a unique complex variety $V$ in $\mathbb{C}^N$ such
that the boundary of $V$ is $X$. Let $\pi: (M, A_1,\cdots,
A_k)\rightarrow (V, 0_1,\cdots, 0_k)$ be a resolution of all the
singularities with $A_i=\pi^{-1}(0_i)$, $1\le i\le k$, as
exceptional sets.

In order to solve the classical complex Plateau problem, we need to
find some $CR$-invariant which can be calculated directly from the
boundary $X$ and the vanishing of this invariant will give
a smooth solution to complex Plateau problem.

For this purpose, we define a new sheaf $\bar{\bar{\Omega}}_V^{1,
1}$, a new invariant of surface singularities $g^{(1,1)}$ and a new $CR$
invariant $g^{(1,1)}(X)$ in \cite{Du-Ya}.  Recently, we
generalized them to higher dimension for dealing with general complex
Plateau problem (\cite{D-G-Y}).

\begin{definition}
\emph{Let $(V, 0)$ be a Stein germ of an $n$-dimensional analytic
space with an isolated singularity at $0$. Suppose $\bar{\bar{\Omega}}^1_V:=\theta_*\Omega^1_{V\backslash V_{sing}}$
where $\theta : V\backslash V_{sing}\longrightarrow V$ is the
inclusion map and $V_{sing}$ is the singular set of $V$. Define a sheaf of germs
$\bar{\bar{\Omega}}_V^{\Lambda^p 1}$ by the sheaf associated with the
presheaf
\[
U\mapsto \langle\Lambda^p\Gamma(U, \bar{\bar{\Omega}}^1_{V})\rangle,
\]
where $U$ is an open set of $V$ and $2 \le p\le n$.}
\end{definition}

\begin{lemma}(\cite{D-G-Y})\label{loc inv2}
\emph{Let $V$ be an $n$-dimensional Stein space with $0$ as its only
singular point in $\mathbb{C}^N$. Let $\pi: (M, A)\rightarrow (V,
0)$ be a resolution of the singularity with $A$ as exceptional set.
Then $\bar{\bar{\Omega}}_V^{\Lambda^p 1}$ is coherent and there is a
short exact sequence
\begin{equation}
0\longrightarrow\bar{\bar{\Omega}}_V^{\Lambda^p
1}\longrightarrow\bar{\bar{\Omega}}_V^p\longrightarrow\mathscr{G}^{(\Lambda^p
1)}\longrightarrow 0
\end{equation}
where $\mathscr{G}^{(\Lambda^p 1)}$ is a sheaf supported on the
singular point of $V$. Let
\begin{equation}
G^{(\Lambda^p 1)}(M\backslash A):=\Gamma(M\backslash A,
\Omega^p_M)/\langle\Lambda^p \Gamma(M\backslash A, \Omega^1_M)\rangle,
\end{equation}
then $\dim \mathscr{G}^{(\Lambda^p 1)}_0=\dim G^{(\Lambda^p
1)}(M\backslash A)$.}
\end{lemma}

Thus, from Lemma \ref{loc inv2}, we can define a local invariant of
a singularity which is independent of resolution.
\begin{definition}
\emph{Let $V$ be an $n$-dimensional Stein space with $0$ as its only
singular point. Let $\pi: (M, A)\rightarrow (V, 0)$ be a resolution
of the singularity with $A$ as exceptional set. Let
\begin{equation}
g^{(\Lambda^p 1)}(0):=\dim \mathscr{G}^{(\Lambda^p 1)}_0=\dim
G^{(\Lambda^p 1)}(M\backslash A).
\end{equation}}
\end{definition}

We will omit $0$ in $g^{(\Lambda^p 1)}(0)$ if there is no confusion
from the context.

Let $\pi: (M, A_1,\cdots, A_k)\rightarrow (V, 0_1,\cdots, 0_k)$ be a
resolution of the singularities with $A_i=\pi^{-1}(0_i)$, $1\le i\le
k$, as exceptional sets. In this case we still let
$$G^{(\Lambda^p
1)}(M\backslash A):=\Gamma(M\backslash A, \Omega^p_M)/\langle\Lambda^p
\Gamma(M\backslash A, \Omega^1_M)\rangle,$$ where $A=\cup_i A_i$.

\begin{definition}
\emph{If $X$ is a compact connected strongly pseudoconvex $CR$
manifold of real dimension $2n-1$, in the boundary of a bounded
strongly pseudoconvex domain $D$ in $\mathbb{C}^N$. Suppose $V$ in
$\mathbb{C}^N$ such that the boundary of $V$ is $X$. Let $\pi: (M,
A_1,\cdots,
A_k)\rightarrow (V, 0_1,\cdots, 0_k)$ be a resolution of
the singularities with $A_i=\pi^{-1}(0_i)$, $1\le i\le k$, as
exceptional sets. Let
\begin{equation}
G^{(\Lambda^p 1)}(M\backslash A):=\Gamma(M\backslash A ,
\Omega^p_M)/\langle\Lambda^p \Gamma(M\backslash A, \Omega^1_M)\rangle
\end{equation}
and
\begin{equation}
G^{(\Lambda^p 1)}(X):=\mathscr{S}^p(X)/\langle\Lambda^p \mathscr{S}^1(X)\rangle,
\end{equation}
where $\mathscr{S}^q$ are holomorphic cross sections of
$\wedge^q(\widehat{T}(X)^*)$. Then we set
\begin{equation}
g^{(\Lambda^p 1)}(M\backslash A):=\dim G^{(\Lambda^p 1)}(M\backslash
A),
\end{equation}
\begin{equation}
g^{(\Lambda^p 1)}(X):=\dim G^{(\Lambda^p 1)}(X).
\end{equation}}
\end{definition}

\begin{remark}
Form the definition (cf. Definition 3.7 and 3.8 in \cite{Du-Ya}), $g^{(\Lambda^2 1)}=g^{(1,1)}$ and
$g^{(\Lambda^2 1)}(X)=g^{(1,1)}(X)$.
\end{remark}

\begin{lemma}(\cite{D-G-Y})\label{boundary}
\emph{Let $X$ be a compact connected strongly pseudoconvex $CR$
manifold of real dimension $2n-1$ which bounds a bounded strongly
pseudoconvex variety $V$ with only isolated singularities
$\{0_1,\cdots, 0_k\}$ in $\mathbb{C}^N$. Let $\pi: (M, A_1,\cdots,
A_k)\rightarrow (V, 0_1,\cdots, 0_k)$ be a resolution of the
singularities with $A_i=\pi^{-1}(0_i)$, $1\le i\le k$, as
exceptional sets. Then $g^{(\Lambda^p 1)}(X)=g^{(\Lambda^p
1)}(M\backslash A)$, where $A=\cup A_i$, $1\le i\le k$.}
\end{lemma}

 By Lemma \ref{boundary} and the proof of Lemma \ref{loc inv2}, we can get the following lemma easily.

\begin{lemma}(\cite{D-G-Y})\label{gp1}
\emph{Let $X$ be a compact connected strongly pseudoconvex $CR$
manifold of real dimension $2n-1$, which bounds a bounded strongly
pseudoconvex variety $V$ with only isolated singularities
$\{0_1,\cdots, 0_k\}$ in $\mathbb{C}^N$. Then $g^{(\Lambda^p
1)}(X)=\sum_i g^{(\Lambda^p 1)}(0_i)=\sum_i \dim
\mathscr{G}^{(\Lambda^p 1)}_{0_i}$.}
\end{lemma}


The following theorem is the crucial part for solving the classical
complex Plateau problem of real dimension $3$.

\begin{theorem}(\cite{Du-Ya})\label{new inv}
\emph{Let $V$ be a $2$-dimensional Stein space with $0$ as its only
normal singular point with $\mathbb{C}^*$-action. Let $\pi: (M,
A)\rightarrow (V, 0)$ be a minimal good resolution of the
singularity with $A$ as exceptional set, then $g^{(\Lambda^2 1)}\ge
1$.}
\end{theorem}

\begin{remark}
We also show that $g^{(\Lambda^2 1)}$ is strictly positive for
rational singularities (\cite{Du-Ga}) and minimal elliptic
singularities (\cite{Du-Ga2}) and exact $1$ for rational double
points, triple points and quotient singularities (\cite{Du-Lu-Ya}).
\end{remark}
Similarly, the following two theorems are the crucial results for solving the
classical complex Plateau problem of real dimension $5$.
\begin{theorem}(\cite{D-G-Y})\label{newnnr}
\emph{Let $V$ be an $n$-dimensional Stein space with $0$ as its only
non-rational singular point, where $n>2$, then $g^{(\Lambda^n 1)}\ge
1$.}
\end{theorem}

\begin{theorem}(\cite{D-G-Y})\label{new3}
\emph{Let $V$ be a $3$-dimensional Stein space with $0$ as its only
normal Gorenstein singular point, then $g^{(\Lambda^3 1)}\ge 1$.}
\end{theorem}

Now we will prove more general result which shows that generalized Yau's conjecture mentioned in Section $1$ is true for isolated local complete intersection singularities of dimension $n\ge 3$. The main idea in the proof is to consider the discrepancy along a special prime exceptional divisor.

\begin{definition}
Suppose that $X$ is a normal variety such that its canonical class
$K_X$ is $\mathbb{Q}$-Cartier, and let $f:Y\rightarrow X$ be a
resolution of the singularities of $X$. Then
\[K_Y=f^*K_X+\sum_ia_iE_i,\]
where the sum is over the irreducible exceptional divisors, and the
$a_i$'s are rational numbers, called the discrepancies.
\end{definition}

\begin{definition}
The minimal discrepancy of a variety $X$ at $0$, denoted by
$Md_0(X)$ (or $Md(X)$ for short), is the minimum of all
discrepancies of discrete valuations of $\mathbb{C}(X)$, whose
center on $X$ is $0$.
\end{definition}
\begin{remark}
The minimal discrepancy only exists when $X$ has log-canonical
singularities (see, e.g. \cite{C-K-M}). Whenever $Md(X)$ exists it
is at least $-1$.
\end{remark}

Shokurov conjecture that the minimal discrepancy is bounded above in
term of the dimension of a variety.

\begin{conjecture}\label{co} (Shokurov \cite{Sh}): The minimal
discrepancy $Md_0(X)$ of a variety $X$ at $0$ of dimension $n$ is at
most $n-1$. Moreover, if $Md_0(X)=n-1$, then $(X, 0)$ is
nonsingular.
\end{conjecture}

The conjecture was confirmed for surfaces (\cite{Al}) and
$3$-dimensional singularities after the explicit classification
(\cite{Re}) of Gorenstein terminal $3$-fold singularities with
\cite{E-M-Y} or \cite{Ma}. If $X$ is a local complete intersection,
then the conjecture also holds (see \cite{E-M} and \cite{E-M-Y}).

%

The following theorem solves generalized Yau's conjecture for isolated local complete intersection singularities and it is also the crucial point for solving the classical Plateua problem in the next section.
\begin{theorem}\label{newn}
\emph{Let $V$ be a $n$-dimensional Stein space with $0$ as its only
isolated local complete intersection singular point, where $n\ge 3$. Then $g^{(\Lambda^n 1)}\ge 1$.}
\end{theorem}
\begin{proof}
We only need to show that the result holds for rational singularities from Theorem \ref{newnnr}.

Let $\pi: M\rightarrow V$ be a resolution such that $E=\cup E_i$ as the exceptional set of $\pi$, where each $E_i$ is the nonsingular irreducible component of dimension $n-1$ with normal crossings. It is well known that isolated local complete intersection singularities are Gorenstein and rational Gorenstein singularities are canonical (see \cite{Ko-Mo} Corollary 5.24). Since canonical singularities are log terminal, by Hacon and Mckernan's result (cf. Corollary 1.5 in \cite{H-M}), $E$ is rational chain connected. So for each $E_i$, $H^0(E_i, \Omega_{E_i}^1)=0$ (cf. \cite{Ko} Corollary 3.8).


We know that Shokurov's minimal discrepancy conjecture holds for local complete intersection singularities (cf. \cite{E-M}, \cite{E-M-Y}), so there exists $s\in
\Gamma(M, \Omega_M^n)$ such that
Ord$_{F} s\le n-2$, where $F$ is some $E_i$. Take a tubular neighborhood
$N$ of $F$ such that $N\subset M$. Consider the exact sequence
(\cite{E-V})
\begin{equation}\label{logEj}
0\rightarrow \Omega_{N}^{1}(\log F)(-F)\rightarrow
\Omega_{N}^{1}\rightarrow \Omega_{F}^{1}\rightarrow 0.
\end{equation}
By taking global sections we have
\begin{equation}\label{loggEj}
0\rightarrow \Gamma({N}, \Omega_{N}^{1}(\log F)(-F))\rightarrow
\Gamma({N}, \Omega_{N}^{1})\rightarrow \Gamma(F,
\Omega_{F}^{1}).
\end{equation}
Since $H^{0}(F,
\Omega^1_{F})=0$,
\begin{equation}\label{equal}
\Gamma(N, \Omega_{N}^{1}(\log F)(-F))=\Gamma(N,
\Omega_{N}^{1})
\end{equation}
from (\ref{loggEj}).

Suppose $\eta\in \Gamma(N, \Omega_{N}^1)$, then $\eta\in
\Gamma(N, \Omega_{N}^1(\log F)(-F))$ by (\ref{equal}). Chose
a point $P$ in $F$ which is a smooth point in $E$. Let $(x_1, x_2,\cdots
x_{n})$ be a coordinate system center at $P$ such that $F$ is given
locally by $x_1=0$. Write $\eta$ locally around $P$ : $\eta\circeq
f_1dx_1+f_2x_1dx_2+\cdots f_nx_1dx_n$, where $f_1$, $f_2$, $\cdots$, $f_n$ are
holomorphic functions and ``$\circeq$" means local equality around
$P$. So the vanishing order of any elements in $\Lambda^n\Gamma(M,
\Omega_M^1)$ along the irreducible exceptional set $F$ is at least
$n-1$ by noticing $\Gamma(M, \Omega_M^1)\subseteq \Gamma(N,
\Omega_{N}^1)$ under natural restriction. Therefore $s \notin\Lambda^n \Gamma(M,
\Omega_M^1)$. Because the singularity is rational,
\[g^{(\Lambda^n 1)}=\dim\Gamma(M, \Omega_M^n)/\Lambda^n \Gamma(M,
\Omega_M^1)\ge 1.\]
\end{proof}

\begin{remark}
If Shokurov¡¯s minimal discrepancy conjecture were true for isolated Gorenstein terminal singularities, then Theorem \ref{newn} is also true for isolated Gorenstein singularities of dimension
$n\ge 3$. Finally, our main Theorem \ref{nd} holds for Calabi-Yau CR manifold of non-
hypersurface type.
\end{remark}

\section{\textbf{The classical complex Plateau problem}}
In 1981, Yau [Ya] solved the classical complex Plateau problem for
the case $n\ge 3$.
\begin{theorem}([Ya])\label{Yau Pla}
\emph{Let $X$ be a compact connected strongly pseudoconvex $CR$
manifold of real dimension $2n-1$, $n\ge 3$, in the boundary of a
bounded strongly pseudoconvex domain $D$ in $\mathbb{C}^{n+1}$. Then
$X$ is the boundary of a complex sub-manifold $V\subset D-X$ if and
only if Kohn--Rossi cohomology groups $H^{p, q}_{K R}(X)$ are zeros
for $1\le q\le n-2$}
\end{theorem}

When $n=2$, Yau and the author
used $CR$ invariant $g^{(1,1)}(X)$ to give the sufficient and
necessary condition for the variety bounded by a Calabi-Yau CR
manifold $X$ being smooth if $H_h^2{(X)}=0$(\cite{Du-Ya}).

\begin{theorem}(\cite{Du-Ya})
\emph{Let $X$ be a strongly pseudoconvex compact Calabi-Yau $CR$
manifold of dimension $3$. Suppose that $X$ is contained in the
boundary of a strongly pseudoconvex bounded domain $D$ in
$\mathbb{C}^N$ with $H_h^2{(X)}=0$. Then $X$ is the boundary of a
complex sub-manifold (up to normalization) $V\subset D-X$ with
boundary regularity if and only if $g^{(1,1)}(X)=0$.}
\end{theorem}

\begin{corollary}(\cite{Du-Ya})\label{h2}
\emph{Let $X$ be a strongly pseudoconvex compact $CR$ manifold of
dimension $3$. Suppose that $X$ is contained in the boundary of a
strongly pseudoconvex bounded domain $D$ in $\mathbb{C}^3$ with
$H_h^2{(X)}=0$. Then $X$ is the boundary of a complex sub-manifold
$V\subset D-X$ if and only if $g^{(1,1)}(X)=0$ .}
\end{corollary}

When $X$ is a Calabi--Yau $CR$ manifold of dimension $5$ ($n=3$), we give
the following necessary and sufficient condition for the variety
bounded by $X$ being smooth in \cite{D-G-Y}.
\begin{theorem}(\cite{D-G-Y})\label{5d}
\emph{Let $X$ be a strongly pseudoconvex compact Calabi-Yau $CR$
manifold of dimension $5$. Suppose that $X$ is contained in the
boundary of a strongly pseudoconvex bounded domain $D$ in
$\mathbb{C}^N$. Then $X$ is the boundary of a complex sub-manifold
(up to normalization) $V\subset D-X$ with boundary regularity if and
only if $g^{(\Lambda^3 1)}(X)=0$.}
\end{theorem}

The main idea of the proof is based on the following theorem and Reid's explicit resolutions of Gorenstein terminal and canonical singularities.

\begin{theorem}(\cite{D-G-Y})\label{rat}
\emph{Let $X$ be a strongly pseudoconvex compact $CR$ manifold of
dimension $2n-1$, where $n>2$. Suppose that $X$ is contained in the
boundary of a strongly pseudoconvex bounded domain $D$ in
$\mathbb{C}^N$. Then $X$ is the boundary of a variety $V\subset D-X$
with boundary regularity and the number of non-rational
singularities is not great than $g^{(\Lambda^n 1)}(X)$. In
particular,  if $g^{(\Lambda^n 1)}(X)=0$, then $V$ has at worst
finite number of rational singularities.}
\end{theorem}


It is a wonderful idea that Yau related complex Plateau problem of hypersurface type to the Kohn-Rossi cohomology groups for $n\ge 3$  in 1981. To determine if $X$ is a boundary of a complex manifold under some conditions, one only needs to calculate a bunch of Kohn-Rossi cohomology groups.
Next theorem shows that we can determine if $X$ is a boundary of a complex manifold by checking only one numerical invariant.
\begin{theorem}\label{nd}
\emph{Let $X$ be a strongly pseudoconvex compact $CR$
manifold of real dimension $2n-1\ge 5$. Suppose that $X$ is contained in the
boundary of a strongly pseudoconvex bounded domain $D$ in
$\mathbb{C}^n$. Then $X$ is the boundary of a complex sub-manifold
 $V\subset D-X$ with boundary regularity if and
only if $g^{(\Lambda^n 1)}(X)=0$.}
\end{theorem}
\begin{proof}
$(\Rightarrow)$ : Since $V$ is smooth, $g^{(\Lambda^n 1)}(X)=0$ follows from
Lemma \ref{gp1}.

$(\Leftarrow)$ : It is well known that $X$ is the boundary of a
variety $V$ in $D$ with boundary regularity (\cite{Lu-Ya},
\cite{Ha-La2}) and the singularities of $V$ are hypersurface singularities. The the result follows easily from Theorem
\ref{newn} and Lemma \ref{gp1}.
\end{proof}

\begin{remark}
In this case, $g^{(\Lambda^n 1)}(X)=0$ is equivalent to vanishing of Kohn-Rossi cohomologies.
\end{remark}

\section*{Acknowledgements}
The author would like to thank  N. Mok for providing excellent research environment in the University of Hong Kong while part of
this research was done there. He would also like to thank Yifei Chen and Jinsong Xu for helpful discussions.



\begin{thebibliography}{Du}
\bibitem[Al]{Al} Alexeev, V., \emph{Two two-dimensional terminations}, Duke
Math. J. 69 (1993), no. 3, 527-545.

\bibitem[B-S-S]{B-S-S} Brasselet J.-P. , Seade J. and Suwa T. \emph{Vector fields on singular varieties}. Lecture Notes in Mathematics, 1987. Springer-Verlag, Berlin, 2009. xx+225 pp.

\bibitem[C-K-M]{C-K-M} Clemens, H., Koll$\acute{a}$r, J. and Mori,
S., \emph{Higher dimensional complex geometry}, Ast$\acute{e}$risque
166 (1988).

\bibitem[Du-Ga]{Du-Ga} Du, R. and Gao, Y., \emph{New invariants for complex manifolds and rational singularities}, Pacific J. of Math., Vol. 269, No. 1 (2014), 73-97.

\bibitem[Du-Ga2]{Du-Ga2} Du, R. and Gao, Y., \emph{Some remarks on Yau's conjecture and complex Plateau problem},  Methods  Appl. Anal., Vol. 21, No. 3 (2014), 357-364.

\bibitem[D-G-Y]{D-G-Y} Du, R. Gao, Y. and Yau, S., \emph{On higher dimensional complex Plateau problem}, Math. Zeit., Vol. 282 (2016), 389-403.

\bibitem[Du-Lu-Ya]{Du-Lu-Ya} Du, R., Luk, H.S. and Yau, S.S.-T.,
\emph{New invariants for complex manifolds and isolated
singularities}, Comm. Anal. Geom., Vol. 19,
No. 5 (2011), 991-1021.

\bibitem[Du-Ya]{Du-Ya} Du, R. and Yau, S.S.-T.,
\emph{Kohn--Rossi cohomology and its application to the complex
Plateau problem, III}, J. Differential Geom., Vol 90, No.
2 (2012), 251-266.

\bibitem[E-M]{E-M} Ein, L. and Musta$\c{t}\check{a}$, M.,
\emph{Inversion of adjunction for local complete intersection
varieties}, Amer. J. Math. 126 (2004), no. 6, 1355-1365.

\bibitem[E-M-Y]{E-M-Y} Ein, L., Musta$\c{t}\check{a}$, M. and Yasuda, T.,
J\emph{et schemes, log discrepancies and inversion of adjunction},
Invent. Math. 153 (2003), no. 3, 519-535.

\bibitem[E-V]{E-V} Esnault, H. and Viehweg, E., \emph{Lectures on vanishing
theorems}, DMV Seminar, 20. Birkh$\ddot{a}$user Verlag, Basel, 1992.
vi+164 pp.


\bibitem[Ha-La]{Ha-La} Harvey, R. and Lawson, B., \emph{On boundaries of
complex analytic varieties I}, Ann. of Math. 102 (1975), 233-290.

\bibitem[Ha-La2]{Ha-La2} Harvey, R. and Lawson, B., \emph{Addendum to
Theorem 10.4 of [HL]}, arXiv: math/0002195

\bibitem[H-M]{H-M} Hacon, C. D. and Mckernan, J., \emph{On Shokurov's rational connectedness conjecture},
Duke Math. J. 138 (2007), no. 1, 119-136.

\bibitem[Ko]{Ko} Koll\'{a}r, J., Rational curves on algebraic varieties, Ergebnisse der Mathematik
und ihrer Grenzgebiete. 3. Folge. A Series of Modern Surveys in Mathematics,
vol. 32, Springer-Verlag, Berlin, 1996.

\bibitem[Ko-Mo]{Ko-Mo} Koll\'{a}r, J. and Mori, S., \emph{Birational geometry of algebraic varieties}, Cam-
bridge Tracts in Mathematics, vol. 134, Cambridge University Press,
Cambridge, 1998, With the collaboration of C. H. Clemens and A.
Corti, Translated from the 1998 Japanese original.


\bibitem[Lu-Ya]{Lu-Ya} Luk, H.S. and Yau,
S.S.-T.,\emph{Kohn--Rossi cohomology
 and its application to the complex Plateau problem, II}, J.
 Differential Geometry, 77 (2007) 135-148, MR2344356.

\bibitem[Ma]{Ma} Markushevich, D., \emph{Minimal discrepancy for a terminal cDV singularity is
1}, J. Math. Sci. Univ. Tokyo 3 (1996), no. 2, 445-456.

%
\bibitem[Re]{Re} Reid, M., \emph{Minimal models of canonical 3-folds}, Algebraic varieties and analytic varieties (Tokyo, 1981), 131¨C180,
Adv. Stud. Pure Math., vol. 1, North-Holland, Amsterdam, 1983.

\bibitem[S-S1]{S-S1} Seade J. and Suwa T., \emph{A residue formula for the index of a holomorphic flow},
Math. Annalen 304 (1996), 621-634.

\bibitem[S-S2]{S-S2} Seade J. and Suwa T., \emph{Residues and topological invariants of singular holomorphic
foliations}, Internat. J. Math. 8 (1997), 825-847.

\bibitem[Sh]{Sh} Shokurov, V.V., \emph{Problems about Fano varieties}, Birational Geometry
of Algebraic Varieties. Open Problems-Katata 1988, 30-32.

\bibitem[Ta]{Ta} Tanaka, N., \emph{A differentail geometry study on strongly pseudoconvex
manifolds}, Lecture in Mathematics, Kyoto University, 9, Kinokuniya
Bookstroe Co. Ltd, 1975.

\bibitem[Ya]{Ya} Yau, S.S.-T., \emph{Kohn--Rossi cohomology and its application to the complex Plateau
problem, I}, Ann. of Math. 113 (1981) 67-110, MR 0604043.

\end{thebibliography}
\end{document}